\documentclass[11pt,twoside]{article}

\usepackage{mathrsfs,amsfonts,amsmath,amssymb}
\usepackage{latexsym}
\usepackage{hyperref}
\newtheorem{satz}{Theorem}
\newtheorem{proposition}[satz]{Proposition}
\newtheorem{theorem}[satz]{Theorem}
\newtheorem{lemma}[satz]{Lemma}
\newtheorem{definition}[satz]{Definition}
\newtheorem{corollary}[satz]{Corollary}
\newtheorem{remark}[satz]{Remark}

\def\no{\noindent}

\def\eps{\varepsilon}
\def\_phi{\varphi}

\def\a{\alpha}

\def\la{\lambda}

\def\F{{\mathbb F}}

\def\o{\omega}
\def\ov{\overline}

\def\C{{\mathbb C}}

\def\E{\mathsf {E}}
\def\T{{\mathbb T}}

\def\Z_N{{\mathbb Z}_N}
\def\Z{{\mathbb Z}}
\def\N{{\mathbb N}}

\def\Gr{{\mathbf G}}

\def\D{{\mathbb D}}

\def\l{\left}
\def\r{\right}

\def\oT{{\rm T}}

\def\G{\Gamma}
\def\FF{\widehat}
\def\c{\circ}
\def\D{\Delta}
\def\Cf{{\mathcal C}}

\def\T{\mathsf {T}}

\author{Shkredov I.D.\footnote{The first and the second author were supported by grant Russian Scientific Foundation RSF 14-11-00433.}, Solodkova E.V., Vyugin I.V.\footnote{The author was supported by grants RFFI N 14-01-00346, President RF MK-4594.2013.1, IUM-Simons Fellowship and Dynasty Foundation.}}
\title{ Intersections of multiplicative subgroups and Heilbronn's exponential sum}
\date{}
\begin{document}
\maketitle

\begin{center}
 Annotation.
\end{center}

{\it {\small

    The paper is devoted to some applications of Stepanov method.
    In the first part of the paper we estimate the cardinality of the specific set, which is obtained as an intersection of additive shifts of several subgroups of $\mathbb{F}_p^*$.
    In the second part we prove
    a new upper bound for Heilbronn's exponential sum and obtain a series of
    applications to the distribution of Fermat quotients. Also we study additive decompositions of multiplicative subgroups.
}}

\section{Introduction}
\label{sec:introduction}

Let $p$ be a prime number, $\mathbb{F}_p=\mathbb{Z}/p\mathbb{Z}$ be a field of $p$ elements, $\mathbb{F}_p^*=\mathbb{F}_p\setminus \{0\},$ and $G\subseteq\mathbb{F}_p^*$ be a multiplicative subgroup. A.~Garcia and J.~F.~Voloch \cite{GV} proved that for an arbitrary subgroup $G\subseteq\mathbb{F}_p^*,$ such that $|G|<(p-1)/((p-1)^{\frac{1}{4}}+1)$ and for any $\mu\in\mathbb{F}_p^*$ the estimate
\begin{eqnarray}\label{GV-est}
|G\cap (G+\mu)|\leqslant 4|G|^{\frac{2}{3}}
\end{eqnarray}
imply. D.~R.~Heath--Brown and S.~V.~Konyagin generalized result (\ref{GV-est}) and gave another proof using the Stepanov method. Futher generalization was given in \cite{V-S} and here we recall the main result of that paper
\begin{theorem}\label{th:Vyu_Sh_add_shifts}
Let $G\subseteq\mathbb{F}_p^*$, be a subgroup of $\mathbb{F}_p^*$, $k>1$ be a positive integer, and $\mu_1,\ldots,\mu_k$ be different nonzero residuals. Let also
$$
32k2^{20k\log (k+1)}\leqslant |G|,\qquad p\geqslant 4k|G|(|G|^{\frac{1}{2k+1}}+1).
$$
Then
$$
|G\cap(G+\mu_1)\cap\ldots\cap(G+\mu_k)|\leqslant 4(k+1)(|G|^{\frac{1}{2k+1}}+1)^{k+1}.
$$
\end{theorem}

Roughly speaking, the corollary above asserts that $|G\cap(G+\mu_1)\cap\ldots\cap(G+\mu_k)|<_k |G|^{\frac{1}{2}+\alpha_k}$, provided by $1\ll_k |G|\ll_k p^{1-\beta_k}$, where $\alpha_k,\beta_k$ are some sequences of positive numbers, and $\alpha_k,\beta_k\to 0$, $k\to\infty$.

In this paper we move ahead and obtain a similiar result for intersections of several multiplicative subgroups

\begin{theorem}\label{th-diff-subgr}
Let $G_0, \ldots, G_k \subseteq \mathbb{F}^{*}_p$ be subgroups of $\mathbb{F}_p^*$, $\mu_1,\ldots,\mu_k$ be distinct nonzero residuals.
Suppose that for all $k_1=1,\ldots,k$ 
one has 
\begin{eqnarray}\label{rest1}
|G_0|\cdot\ldots\cdot|G_{k_1}|<(k_1+2)^{-\frac{2k_1+1}{2}}p^{k_1+\frac{1}{2}}
\end{eqnarray}
and for all $k_1=0,\ldots,k$ and for all $j=0,\ldots,k_1$ the following 
holds 
\begin{eqnarray}\label{rest2}
\frac{1}{2}\left(\prod_{i=0}^{k_1} |G_i|\right)^{\frac{1}{2k_1+1}}<|G_j|<\frac{1}{2(k_1+3)}\left(\prod_{i=0}^{k_1} |G_i|\right)^{\frac{2}{2k_1+1}} \,.
\end{eqnarray}
Then
$$
|G_0\cap (G_1+\mu_1)\cap\ldots\cap (G_k+\mu_k)|\leqslant 4k(k+2)(|G_0||G_1|\ldots|G_k|)^{\frac{1}{2k+1}}.
$$
\end{theorem}

A particular case $k=1$ of the theorem above allows us to get a new result on additive decomposition of multiplicative subgroups, see  Section~\ref{sec:subgroup_decomp}.

Sections~\ref{sec:stepanov}, \ref{sec:proof} are devoted to Heilbronn's exponential sum.
Heilbronn's exponential sum is defined by
\begin{equation}\label{def:Heilbronn_sum}
	S(a) = \sum_{n=1}^p e^{2 \pi i \cdot \frac{an^p}{p^2} }\, .
\end{equation}

D.R. Heath--Brown obtained in~\cite{H} the first nontrivial upper bound for the sum. This result was improved in papers~\cite{H-K}, \cite{s_heilbronn}, \cite{s_heilbronn-ms} (see also \cite{Yu}).
Let us formulate, for example, the main result from \cite{s_heilbronn}.

\begin{theorem}
    Let $p$ be a prime, and $a\neq 0 \pmod p$.
    Then
    $$
        |S(a)| \ll p^{\frac{5}{6}} \log^{\frac{1}{6}} p \,.
    $$
\label{t:main-}
\end{theorem}

The main result of Section~\ref{sec:proof} is the following.

\begin{theorem}
    Let $p$ be a prime.
    Then
    $$
        \sum_a |S(a)|^4 \ll p^{\frac{58}{13}} \log^{\frac{14}{13}} p \,.
    $$
\label{t:main+}
\end{theorem}

Thus, we obtain better bound for $L_4$--norm of $S(a)$ but not for individual sum.
The results of such a sort are useful in applications, see below.

Heilbronn's exponential  sum is connected
(see e.g. \cite{BFKS}, \cite{Chang_Fermat}, \cite{Lenstra}, \cite{OstShp}, \cite{Shp-FermVal}, \cite{Shp-Ihara})
 with so--called {\it Fermat quotients} defined as
$$
    q(n) = \frac{n^{p-1}-1}{p} \,, \quad n\neq 0 \pmod p \,.
$$
Our main result has some applications
to the distribution of such quotients.
The list of the applications can be found in \cite{s_heilbronn} (see also Section~\ref{sec:proof}).

Our approach can be described as follows. To obtain Theorem \ref{th-diff-subgr} we just generalize the method from \cite{V-S} and make more accurate calculations. 
Instead of key Lemma 3.4 from \cite{V-S} we use a general result of F.K. Schmidt (see Lemma \ref{lem:schmidt_wronskian_nonzero}) on linear dependence over $\F((x))$.
As for Theorem~\ref{t:main+} then, clearly, sum (\ref{def:Heilbronn_sum}) can be considered as the sum over
the following multiplicative subgroup
\begin{equation}\label{def:H_Gamma}
    \G = \{ m^p ~:~ 1\le m \le p-1 \} \subseteq \Z/(p^2 \Z) 
\end{equation}
(see the discussion at the beginning of Section~\ref{sec:stepanov}).
Recently, some progress in estimating of  exponential sums over ``large''\, subgroups
(but in $\Z/p\Z$ not in $\Z/p^2\Z$)
such as
(\ref{def:H_Gamma})
was attained (see \cite{s_ineq}).
So, it is natural to use the approach from the paper to obtain a new upper bound for (\ref{def:Heilbronn_sum}).
Applying Stepanov's method  (see Section~\ref{sec:stepanov})
as well as some combinatorial observations (see Lemma \ref{l:ineq}),
we estimate ``the additive energy''\, of the subgroup $\G$.
This new bound easily implies
our Theorem \ref{t:main+}.

We are going to obtain some new facts about distribution of the elements of Heilbronn's subgroup $\G$ in the future.

The authors are grateful to Sergey Konyagin for useful discussions.
I.D.S. is grateful to Yuri Shteinikov for pointing to him a mistake in calculations
in the first version of the paper and very useful discussions.

\section{Definitions}
\label{sec:definitions}

Let $\Gr$ be an abelian group.
If $\Gr$ is finite then denote by $N$ the cardinality of $\Gr$.
It is well--known~\cite{Rudin_book} that the dual group $\FF{\Gr}$ is isomorphic to $\Gr$ in the case.
Let $f$ be a function from $\Gr$ to $\mathbb{C}.$  We denote the Fourier transform of $f$ by~$\FF{f},$
\begin{equation}\label{F:Fourier}
  \FF{f}(\xi) =  \sum_{x \in \Gr} f(x) e( -\xi \cdot x) \,,
\end{equation}
where $e(x) = e^{2\pi i x}$.
We rely on the following basic identities
\begin{equation}\label{F_Par}
    \sum_{x\in \Gr} |f(x)|^2
        =
            \frac{1}{N} \sum_{\xi \in \FF{\Gr}} \big|\widehat{f} (\xi)\big|^2 \,.
\end{equation}
\begin{equation}\label{svertka}
    \sum_{y\in \Gr} \Big|\sum_{x\in \Gr} f(x) g(y-x) \Big|^2
        = \frac{1}{N} \sum_{\xi \in \FF{\Gr}} \big|\widehat{f} (\xi)\big|^2 \big|\widehat{g} (\xi)\big|^2 \,.
\end{equation}
and
\begin{equation}\label{f:inverse}
    f(x) = \frac{1}{N} \sum_{\xi \in \FF{\Gr}} \FF{f}(\xi) e(\xi \cdot x) \,.
\end{equation}
If
$$
    (f*g) (x) := \sum_{y\in \Gr} f(y) g(x-y) \quad \mbox{ and } \quad (f\circ g) (x) := \sum_{y\in \Gr} f(y) g(y+x)
$$
 then
\begin{equation}\label{f:F_svertka}
    \FF{f*g} = \FF{f} \FF{g} \quad \mbox{ and } \quad \FF{f \circ g} = \FF{f}^c \FF{g} = \ov{\FF{\ov{f}}} \FF{g} \,,
\end{equation}
where for a function $f:\Gr \to \mathbb{C}$ we put $f^c (x):= f(-x)$.
 Clearly,  $(f*g) (x) = (g*f) (x)$ and $(f\c g)(x) = (g \c f) (-x)$, $x\in \Gr$.
 The $k$--fold convolution, $k\in \N$  we denote by $*_k$,
 so $*_k := *(*_{k-1})$.

We use in the paper  the same letter to denote a set
$S\subseteq \Gr$ and its characteristic function $S:\Gr\rightarrow \{0,1\}.$
Write $\E(A,B)$ for the {\it additive energy} of two sets $A,B \subseteq \Gr$
(see e.g. \cite{tv}), that is
$$
    \E(A,B) = |\{ a_1+b_1 = a_2+b_2 ~:~ a_1,a_2 \in A,\, b_1,b_2 \in B \}| \,.
$$
If $A=B$ we simply write $\E(A)$ instead of $\E(A,A).$
Clearly,
\begin{equation}\label{f:energy_convolution}
    \E(A,B) = \sum_x (A*B) (x)^2 = \sum_x (A \circ B) (x)^2 = \sum_x (A \circ A) (x) (B \circ B) (x)
    \,,
\end{equation}
and by (\ref{svertka}), we have
\begin{equation}\label{f:energy_Fourier}
    \E(A,B) = \frac{1}{N} \sum_{\xi} |\FF{A} (\xi)|^2 |\FF{B} (\xi)|^2 \,.
\end{equation}

Put for any $A\subseteq \Gr$
$$
   \T_k (A) := | \{ a_1 + \dots + a_k = a'_1 + \dots + a'_k  ~:~ a_1, \dots, a_k, a'_1,\dots,a'_k \in A \} | \,.
$$

Let
\begin{equation}\label{f:E_k_preliminalies}
    \E_k(A)=\sum_{x\in \Gr} (A\c A)(x)^k \,,
\end{equation}
and
\begin{equation}\label{f:E_k_preliminalies_B}
\E_k(A,B)=\sum_{x\in \Gr} (A\c A)(x) (B\c B)(x)^{k-1}
    =\E(\Delta_k (A),B^{k}) \,,
\end{equation}
be the higher energies of $A$ and $B$.
Here
$$
    \Delta (A) = \Delta_k (A) := \{ (a,a, \dots, a)\in A^k \}\,.
$$
Similarly, we write $\E_k(f,g)$ for any complex functions $f$ and $g$.
Put also
$$
    \E^*_{k+1} (A,B)=\sum_{x \neq 0} (A\c A)(x) (B\c B)(x)^{k-1} \,.
$$
Quantities $\E_k (A,B)$ can be
expressed
in terms of generalized convolutions (see \cite{ss_E_k}).

\begin{definition}
   Let $k\ge 2$ be a positive number, and $f_0,\dots,f_{k-1} : \Gr \to \C$ be functions.
Denote by
${\mathcal C}_k (f_0,\dots,f_{k-1}) (x_1,\dots, x_{k-1})$
the function
$$
		\Cf_k (f_0,\dots,f_{k-1}) (x_1,\dots, x_{k-1}) = \sum_z f_0 (z) f_1 (z+x_1) \dots f_{k-1} (z+x_{k-1}) \,.
$$
Thus, $\Cf_2 (f_1,f_2) (x) = (f_1 \circ f_2) (x)$.
If $f_1=\dots=f_k=f$ then write
$\Cf_k (f) (x_1,\dots, x_{k-1})$ for $\Cf_k (f_1,\dots,f_{k}) (x_1,\dots, x_{k-1})$.
\end{definition}

For a positive integer $n,$ we set $[n]=\{1,\ldots,n\}$.
All logarithms used in the paper are to base $2.$
By  $\ll$ and $\gg$ we denote the usual Vinogradov's symbols.
If $N$ is a
positive integer then write $\Z_N$ for $\Z/N\Z$ and
$\Z_N^*$ for the subgroup of all invertible elements of $\Z_N$.

\section{An intersection of additive shifts of subgroups of $\mathbb{F}_p^*$}

{\it Proof of the Theorem \ref{th-diff-subgr}.} Denote
$$
\Omega=G_0\cap (G_1+\mu_1)\cap\ldots\cap (G_k+\mu_k)
$$
and $|G_0|=t_0,\ldots,|G_k|=t_k$.

We will estimate $|\Omega|$ by means Stepanov method. We aim to find the non-zero polynomial
$$
\Psi(x) = \sum_{\mathbf{a}, d} C_{\mathbf{a},d}x^d x^{a_0 t_0}(x-\mu_1)^{a_1 t_1}\cdots(x-\mu_k)^{a_k t_k},
$$
where $\mathbf{a}=(a_0,\ldots,a_k)$, $a_i<B_i$, $d<D$, $i=\overline{0,k},$ such that its coefficients $C_{\mathbf{a},d}$ do not vanish simultaneously and all derivatives
\begin{equation}\label{deriv}
\frac{d^n}{dx^n}\Psi(x) \Bigr|_{x \in \Omega}=0, \quad\quad n=\overline{0,M-1}
\end{equation}
of orders from $0$ to $M-1$ vanish at every $x\in\Omega$.

Suppose that $x\in\Omega$ and $x\not=\mu_i$, $i=\overline{1,k}$, then condition~(\ref{deriv}) is equivalent to
$$
\Bigl[x(x-\mu_1) \ldots (x-\mu_k)\Bigr]^n \frac{d^n}{dx^n}\Psi(x)\Bigr|_{x \in \Omega}=0.
$$
Note that
\begin{eqnarray*}
[x(x-\mu_1)\cdots(x-\mu_k)]^n\frac{d^n}{dx^n}\Bigl(x^d x^{a_0 t_0}(x-\mu_1)^{a_1 t_1}\cdots(x-\mu_k)^{a_k t_k}\Bigr)=\\ x^{a_0 t_0}(x-\mu_1)^{a_1 t_1}\cdots(x-\mu_k)^{a_k t_k}P_{n,\mathbf{a},d}(x),
\end{eqnarray*}
and $P_{n,\mathbf{a},d}(x)$ is either trivial or $\deg P_{n,\mathbf{a},d}(x) \leqslant D+kn$. Note that if $x\in \Omega$
then
$$
x^{t_0}=(x-\mu_1)^{t_1}= \ldots =(x-\mu_k)^{t_k}=1.
$$
Therefore,
$$
\Bigl[x(x-\mu_1) \ldots (x-\mu_k)\Bigr]^n \frac{d^n}{dx^n} \Bigl(x^{d} x^{a_0 t_0}(x-\mu_1)^{a_1 t_1}\cdots(x-\mu_k)^{a_k t_k}\Bigr) \Bigl |_{x \in \Omega}=P_{n,\mathbf{a},d}(x),
$$
and
$$
\Bigl[x(x-\mu_1) \ldots (x-\mu_k)\Bigr]^n \frac{d^n}{dx^n}\Psi(x) \Bigl |_{x \in \Omega}=\sum_{\mathbf{a},d} C_{\mathbf{a},d} P_{n,\mathbf{a},d}(x)=P_n(x).
$$
Now we choose coefficients $C_{\mathbf{a},d}$ to make polynomials $P_n(x)$ zero for all $n<M$. It can be done because the coefficients of polynomials $P_n(x)$ are homogeneous linear forms of coefficients $C_{\mathbf{a},d}$ and the condition
$$
\forall n=\overline{0,M-1}\quad P_n(x)\equiv 0
$$
is equivalent to a system of homogeneous linear equations, which has a nonzero solution if the number of variables $C_{\mathbf{a},d}$ is more than the number of equations (the number of equations is equal to the number of coefficients of polynomials $P_n(x)$, $n<M$).
Consequently, the following
\begin{eqnarray}\label{cond-var}
MD+k\frac{M^2}{2} < DB_0B_1 \ldots B_k
\end{eqnarray}
is sufficient.

If $\Psi(x)$ does not vanish identically then
\begin{eqnarray}\label{est-omega}
|\Omega| \le \frac{\deg \Psi(x)}{M}.
\end{eqnarray}
The following lemma shows that $\Psi(x)$ is not identically zero.

\begin{lemma}\label{lemma-indep}
Let $k,t_0,\ldots,t_k$ be positive integers such that for all $k_1=1,\ldots,k$
\begin{eqnarray}\label{eq:est-deg}
\prod_{i=0}^{k_1} t_i<(k_1+2)^{-k_1-\frac{1}{2}}p^{k_1+\frac{1}{2}},
\end{eqnarray}
and for all $k_1=0,\ldots,k$ and for all $j=0,\ldots,k_1$ the following restrictions
\begin{eqnarray}\label{eq:indep_conditions}
\frac{1}{2}\left(\prod_{i=0}^{k_1} t_i\right)^{\frac{1}{2k_1+1}}<t_j<\frac{1}{2(k_1+3)}\left(\prod_{i=0}^{k_1} t_i\right)^{\frac{2}{2k_1+1}}
\end{eqnarray} imply.
Let
$$
\tau=\left(\prod_{i=0}^k t_i \right)^{\frac{2}{2k+1}},
$$
and define for all $j=\overline{0,k}$: $B_j=\lfloor\tau/t_j\rfloor$, $D=\left\lfloor\frac{1}{2}\prod_{i=0}^k B_i\right\rfloor$.

Then the polynomials
\begin{equation}\label{poly}
x^d x^{a_0 t_0}(x-\mu_1)^{a_1 t_1}\ldots(x-\mu_k)^{a_k t_k},
\end{equation}
where $a_i < B_i$, $d <D$, $i=\overline{1,k}$ are linearly independent over the field $\mathbb{F}_p$.
\end{lemma}

\begin{proof}
First of all it is easy to check that $t_j>d$, $j=0,\ldots,k$.
Suppose, to the contrary, that polynomials (\ref{poly}) are linearly dependent. Then there exists a nontrivial polynomial
\begin{eqnarray*}
\widetilde{\Psi}(x)=\sum \widetilde{C}_{\mathbf{a},d} x^d x^{a_0 t_0}(x-\mu_1)^{a_1 t_1}\ldots(x-\mu_k)^{a_k t_k} \equiv 0.
\end{eqnarray*}
Let us rewrite it in the following form
\begin{eqnarray}\label{zero-comb2}
(x-\mu_k)^{t_k}\sum_{\mathbf{a}:a_k \ne 0}\widetilde{C}_{\mathbf{a},d} x^d x^{a_0 t_0}(x-\mu_1)^{a_1 t_1}\ldots(x-\mu_k)^{(a_k-1) t_k}+\\ +\sum_{\mathbf{a}: a_k = 0}\widetilde{C}_{\mathbf{a},d} x^d x^{a_0 t_0}(x-\mu_1)^{a_1 t_1}\ldots(x-\mu_{k-1})^{a_{k-1} t_{k-1}}=0.\nonumber
\end{eqnarray}
Consider the polynomial
$$
\Phi(x)=\sum_{\mathbf{a}: a_k = 0}\widetilde{C}_{\mathbf{a},d} x^{d} x^{a_0 t_0}(x-\mu_1)^{a_1 t_1}\ldots(x-\mu_{k-1})^{a_{k-1} t_{k-1}}.
$$
$\Phi(x)$ is divided by $(x-\mu_k)^{t_k}$ because of (\ref{zero-comb2}).

Now suppose that the products
\begin{eqnarray}\label{syst-prod}
x^{d} x^{a_0 t_0}(x-\mu_1)^{a_1 t_1}\ldots(x-\mu_{k-1})^{a_{k-1} t_{k-1}},\qquad a_i<B_i,\quad i=\overline{0,k-1},\quad d<D
\end{eqnarray}
are linearly independent. If not, then we can begin a proof of Lemma \ref{lemma-indep} with $k:=k-1$. Indeed, if $k'=k-1$, $t_i'=t_i$ $i=\overline{0,k'}$ then condition \eqref{eq:est-deg} takes place and the conditions \eqref{eq:indep_conditions} hold as well. For $k=0$ the result is trivial.

Consequently, we can suppose that a polynomial $\Phi(x)$ is nonzero.

Rewrite $\Phi(x)$ in the form
$$
\Phi(x)=\sum_{\mathbf{a}:a_k=0}H_{\mathbf{a}}(x)x^{a_0 t_0}(x-\mu_1)^{a_1 t_1}\ldots(x-\mu_{k-1})^{a_{k-1} t_{k-1}},
$$
where $H_{\mathbf{a}}(x)=\sum_{d}\widetilde{C}_{\mathbf{a},d} x^{d}$, all vectors $\mathbf{a}$ are pairwise distinct, and $a_i\in \{ 0,\ldots, B_i-1\}$, $i=\overline{0,k-1}$. We have $\deg H_{\mathbf{a}}(x)< D$ for all $\mathbf{a}$.

Denote by $Q_{\mathbf{a}}(x)$ the following expression
$$
Q_{\mathbf{a}}(x)=H_{\mathbf{a}}(x)x^{a_0 t_0}(x-\mu_1)^{a_1 t_1}\ldots(x-\mu_{k-1})^{a_{k-1} t_{k-1}},
$$
and denote with some abuse of notation by $\mathbf{a}$ the vector $(a_0,\ldots,a_{k-1})$. It is easy to see that polynomials $Q_{\mathbf{a}}(x)$ are linearly independent, because we had already supposed that the polynomials (\ref{syst-prod}) are linearly independent, further, $Q_{\mathbf{a}}(x)$ are linear combinations of polynomials (\ref{syst-prod}), and $t_j>D$, $j=0,\ldots,k-1$.

Consider the Wronskian
\begin{equation}\label{eq:wronskian}
W(x)=\begin{vmatrix}
Q_{(0,\ldots,0)}(x) &\ldots & Q_{(B_0-1,\ldots,B_{k-1}-1)}(x)\\
Q_{(0,\ldots,0)}^{'}(x) & \ldots & Q_{(B_0-1,\ldots,B_{k-1}-1)}^{'}(x)\\
\vdots & \ddots & \vdots\\
Q_{(0,\ldots,0)}^{(B_0 B_1 \ldots B_{k-1}-1)}(x) &\ldots & Q_{(B_0-1,\ldots,B_{k-1}-1)}^{(B_0 B_1 \ldots B_{k-1}-1)}(x)
\end{vmatrix}.
\end{equation}

If the Wronskian is identically zero, then $Q_\mathbf{a}(x)$ are linearly dependent and we can reduce the number of brackets in the initial problem. This fact can be concluded from the result of F.~K.~Schmidt \cite{Shm} (see also \cite{GV-Wron}), which we reformulate for our special case:
\begin{lemma}[F.~K.~Schmidt]\label{lem:schmidt_wronskian_nonzero}
Let $\mathbb{F}$ be a field of characteristic $p>0$ and $f_1,\ldots,f_n \in \mathbb{F}((x)).$ Then $f_1,\ldots,f_n$ are linearly independent over $\mathbb{F}((x^p))$ if and only if the Wronskian $W(f_1,\ldots,f_n) \not\equiv 0.$
\end{lemma}

This lemma claims that if the Wronskian $W(x)$ constructed for the set of polynomials $Q_\mathbf{a}(x)$ is identically zero, then there exists a linear combination of functions
$$\sum_{\mathbf{a}}\varphi_\mathbf{a}(x)Q_\mathbf{a}(x) \equiv 0,$$
which is identically zero, but has at least one nonzero coefficient $\varphi_\mathbf{a}(x) \in \mathbb{Z}_p((x^p)).$ In our case, if every polynomial $Q_\mathbf{a}(x)$ has degree less than the field characteristic $p,$ then Schmidt's result implies the linear dependence of $Q_\mathbf{a}(x).$
Assumption (\ref{eq:est-deg}) gives us that the degrees of polynomials $Q_{\mathbf{a}}(x)$ are less than $p$, so $Q_{\mathbf{a}}(x)$ are linearly independent and the Wronskian~\eqref{eq:wronskian} is not vanished.

As polynomial $W(x)$ is divisible by
$$
R(x)=\prod_{\mathbf{a}}\frac{ x^{a_0t_0}(x-\mu_1)^{a_1t_1}\ldots(x-\mu_{k-1})^{a_{k-1}t_{k-1}}}{(x(x-\mu_1)\ldots(x-\mu_{k-1}))^{B_0 \ldots B_{k-1}-1}}
$$
because $t_i > B_1 \ldots B_{k-1}$ for all $i$ and for all $\mathbf{a}$ the column indexed by $\mathbf{a}$ is divisible by fraction $$\frac{ x^{a_0t_0}(x-\mu_1)^{a_1t_1}\ldots(x-\mu_{k-1})^{a_{k-1}t_{k-1}}}{(x(x-\mu_1)\ldots(x-\mu_{k-1}))^{B_0 \ldots B_{k-1}-1}}.$$
Hence, we have
\begin{equation}
\deg (W(x)/R(x)) \le DB_0B_1 \ldots B_{k-1} + k\frac{(B_0 \ldots B_{k-1})^2}{2}.
\end{equation}
We know that $(x-\mu_k)^{t_k}$ divides $\Phi(x)$, consequently $(x-\mu_k)^{t_k-(B_0 \ldots B_{k-1}-1)}$ divides $W(x)$, because one of of this determinant's columns is divided by $(x-\mu_k)^{t_k-(B_0 \ldots B_{k-1}-1)}$. The order of the root $x=\mu_k$ does not exceed the degree of the polynomial, in other words
$$
t_k-(B_0 \ldots B_{k-1}-1) \le DB_0B_1 \ldots B_{k-1} + k\frac{(B_0 \ldots B_{k-1})^2}{2},
$$
Consequently, if
\begin{eqnarray}\label{ineq-tk}
t_k > DB_0B_1 \ldots B_{k-1} + k\frac{(B_0 \ldots B_{k-1})^2}{2}+B_0 \ldots B_{k-1}-1
\end{eqnarray}
then the polynomials (\ref{syst-prod}) are linearly independent.

Put $\gamma=\min_{0\leqslant i\leqslant k} (1-t_i/\tau)$.
By the definition of the number $\gamma$ we have $t_j>\gamma\frac{\tau}{B_k}$, $j=0,\ldots,k$.
The second inequality from (\ref{eq:indep_conditions}) gives us $\gamma>1-\frac{1}{2(k+3)}$, and $B_j>2(k+3)$, $j=0,\ldots,k$. Therefore
\begin{equation*}
\begin{split}
&DB_0B_1 \ldots B_{k-1} + k\frac{(B_0 \ldots B_{k-1})^2}{2}+B_0 \ldots B_{k-1}-1<\frac{1}{B_k}\left(\frac{1}{2}+\frac{k}{4(k+3)}+\frac{1}{(2k+6)^{k+1}}\right)\left(\prod_{i=0}^kB_i\right)^2\\&<\frac{1}{B_k}\left(1-\frac{1}{2(k+3)}\right)\left(\prod_{i=0}^kB_i\right)^2<\gamma\frac{\tau}{ B_k}<t_k.
\end{split}
\end{equation*}
Hence the inequality (\ref{ineq-tk}) holds and the claim follows.
\end{proof}

Let us return to the proof of Theorem \ref{th-diff-subgr}.

Take the parameters $\tau$ and $B_i$, $i=\overline{0,k};$ $D, \gamma$ as in Lemma \ref{lemma-indep}. We have
\begin{equation}
\gamma\frac{\tau}{t_i}<B_i\leqslant\frac{\tau}{t_i}
\end{equation}
and $\gamma>1-\frac{1}{2(k+3)},$ so
\begin{eqnarray}\label{gammaexp}
\gamma^{k+2}>\left(1-\frac{1}{2(k+3)}\right)^{k+2}>1/\sqrt{e}>1/2.
\end{eqnarray}
Take the multiplicity parameter $M$ as
\begin{eqnarray}\label{defM}
M=\left\lfloor\frac{1}{2k}\prod_{i=0}^k B_i\right\rfloor.
\end{eqnarray}
Using the definition of $\gamma$ several times, we have the following estimate for $M$
\begin{equation}
M=\left\lfloor\frac{1}{2k}\prod_{i=0}^k B_i\right\rfloor\geqslant \left\lfloor\frac{\gamma^{k+1}}{2k}\frac{\tau^{k+1}}{\prod_{i=0}^kt_i}\right\rfloor\geqslant \frac{\gamma^{k+2}}{2k}\left(\prod_{i=0}^k t_i\right)^{\frac{1}{2k+1}}\geqslant \frac{1}{4k}\left(\prod_{i=0}^k t_i\right)^{\frac{1}{2k+1}}.
\end{equation}
It is easy to see that inequality (\ref{cond-var}) holds. Indeed, by (\ref{defM}) and $\gamma>1-\frac{1}{2(k+3)}$ we get
\begin{equation}
\begin{split}
&MD+k\frac{M^2}{2}<\frac{1}{4k}\left(\prod_{i=0}^kB_i\right)^2+\frac{1}{8k}\left(\prod_{i=0}^kB_i\right)^2=\frac{3}{8k}\left(\prod_{i=0}^kB_i\right)^2\\&<\frac{\gamma}{2}\left(\prod_{i=0}^kB_i\right)^2  <\left\lfloor\frac{1}{2}\left(\prod_{i=0}^kB_i\right)\right\rfloor\left(\prod_{i=0}^k B_i\right)=D\left(\prod_{i=0}^kB_i\right).
\end{split}
\end{equation}
Returning to (\ref{est-omega}) and applying (\ref{defM}), we obtain an estimate
\begin{equation}
|\Omega|\leqslant \frac{\deg \Phi(x)}{M}<\frac{(k+2)\tau}{\frac{\gamma^{k+2}}{2k}\left(\prod_{i=0}^k t_i\right)^{\frac{1}{2k+1}}}<\frac{2k(k+2)}{\gamma^{k+2}}\left(\prod_{i=0}^k t_i\right)^{\frac{1}{2k+1}}
\end{equation}
and by (\ref{gammaexp}) we have
\begin{equation}
|\Omega|<4k(k+2)\left(\prod_{i=0}^k t_i\right)^{\frac{1}{2k+1}}.
\end{equation}
This completes the proof.

\bigskip Note that the result of Theorem \ref{th-diff-subgr} can be easily extended to the case of different cosets of subgroups $G_0,G_1,\ldots,G_k$, see \cite{V-S} for details.

We can rewrite the last theorem in the following form.

\begin{corollary}
Consider a system of equations
\begin{eqnarray}\label{syst-equ}
(x-\mu_i)^{\lambda_i}=1,\quad i=0,\ldots,k, \quad x\in\mathbb{F}_p
\end{eqnarray}
with arbitrary pairwise distinct $\mu_i \in \mathbb{F}_p^{*}$, $i=0,\ldots,k,$ where $\lambda_i\mid (p-1)$, $i=\overline{0,k}$ and $p, t_i=(p-1)/\lambda_i$, $i=\overline{0,k}$, satisfy the conditions of Lemma \ref{lemma-indep}. Then
the number of solutions of system (\ref{syst-equ}) does not exceed
$$
4k(k+2)\left(\prod_{i=0}^k \frac{p}{\lambda_i}\right)^{\frac{1}{2k+1}}.
$$
\end{corollary}

\section{Additive decomposition of small subgroups}\label{sec:subgroup_decomp}

Let $S$ be a subset in an abelian group $\mathbf{G}.$ We say that $S$ is \textit{reducible} or additively decomposed if it can be represented as
$$S=A+B,$$
where $A,B \subseteq \mathbf{G}$ are arbitrary and the sumset $A+B$ is defined by $A+B:=\{a+b: a \in A, b \in B\}.$ We call an additive decomposition \textit{nontrivial}, if $|A| \ge 2$ and $|B| \ge 2.$

In this section we show how Theorem~\ref{th-diff-subgr} and similiar results can be applied to the problem of reducibility of small subgroups in $\mathbb{F}_p.$

For $k=1$ Theorem \ref{th-diff-subgr} gives us the following

\begin{corollary}\label{cor:k_eq_1}
Let $G_0$, $G_1$ be two subgroups of $\mathbb{F}_p^*$ and $\mu \in \mathbb{F}_p^{*}$. If
\begin{equation*}
|G_0||G_1|<p^{3/2}3^{-3/2}, |G_0|>6,\qquad |G_0|>\frac{1}{512}|G_1|^2,\qquad |G_1|<\frac{1}{512}|G_0|^2,
\end{equation*}
then
$$
|G_0\cap (G_1+\mu)|\leqslant 12(|G_0||G_1|)^{1/3}.
$$
\end{corollary}
A similar result was obtained by Mit'kin~\cite{Mitkin}.
\begin{lemma}[Mit'kin]\label{lem:mitkin}
Let $p>2$ be a prime, $\Gamma,\Pi$ be subgroups of $\mathbb{F}_p^*,$ $M_\Gamma,M_\Pi$ be sets of distinct coset representatives of $\Gamma$ and $\Pi$ respectively. For an arbitrary set $\Theta \subset M_\Gamma \times M_\Pi$ such that $(|\Gamma||\Pi|)^2|\Theta| < p^3$ and $|\Theta| \le 33^{-3}|\Gamma||\Pi|$ we have
\begin{equation}\label{eq:mitkin}
\sum_{(u,v) \in \Theta}\Bigl|\{(x,y) \in \Gamma \times \Pi : ux+vy=1\}\Bigr| \ll (|\Gamma||\Pi||\Theta|^2)^{1/3}.
\end{equation}
\end{lemma}

Lemma above gives us an upper bound for the mean of convolution of two different subgroups $\Gamma$ and $\Pi$. In the case $\Gamma=\Pi$ similar estimate was obtained by Konyagin in~\cite{K_Tula}, see Proposition~\ref{p:H-K_2/3} below. Combining the method from~\cite{K_Tula} with our technique one can get an analog of~\eqref{eq:mitkin}.

Now we can formulate the main result of this section.

\begin{theorem}\label{th:decomposition_small_subgroups}
Let $\eps \in (0,1]$ be a real number, $A, G \subset \mathbb{F}_p^{*}$ be sufficiently large multiplicative subgroups and $B \subseteq \mathbb{F}_p$ be an arbitrary nonempty set. If $|G \cap A| \ll |A|^{1-\eps},$ $|G|^2|A|^{1+\eps}|B| \ll p^3,$ $|G|^2|A|^2 \ll p^3$ and $A+B \subseteq G,$ then $|B||A|^{1+\eps} \ll |G|.$
\end{theorem}

\begin{proof}
Note that $B$ cannot contain zero, because otherwise $A \subseteq G$ and $|G \cap A|=|A|$ which is denied by the conditions of the theorem. Let $H=G \cap A.$ Clearly, $H$ is a multiplicative subgroup. Let $B_\xi=B \cap \xi H,$ $\xi \in \mathbb{F}^{*}_p/H.$ Each $B_\xi$ contains at most $|H|$ elements, hence, there exist $k=\lceil |B|/|A|^{1-\eps}\rceil$ nonzero elements from $B$ such that $b_i \not\equiv b_j\;(\bmod\; H).$

Let $U=\{(1/b_i, -1/b_i) : b_i \in B, i=1,\ldots,k;\quad b_i \not\equiv b_j\;(\bmod\; H)\mbox{ for }i \neq j\}.$ One can see that if $M_G$ and $M_A$ are sets of distinct coset representatives of $G$ and $A$ respectively, then $U \subseteq M_G \times M_A.$ Indeed, if $1/b_i \equiv 1/b_j\;(\bmod\; G)$ and $-1/b_i=-1/b_j\;(\bmod\; A)$ for some $i \neq j,$ then $b_j/b_i \in H,$ which is impossible by the definition of $U.$

Now we can use Lemma~\ref{lem:mitkin} with $\Gamma=G, \Pi=A$ and $\Theta=U.$ If $G,A$ and $B$ satisfy the conditions of our theorem, then $G,A$ and $U$ satisfy the conditions of the Lemma. As a result we obtain an estimate
\begin{equation}\label{eq:conv_est_mitkin}
\sum_{i=1}^k(A \circ G)(b_i) \ll (|G||A|k^2)^{1/3}.
\end{equation}
Observe that if $A+B \subseteq G,$ then $\sum_{i=1}^k(A \circ G)(b_i)=k|A|,$ so $k \ll |G|/|A|^2$ and the claim follows.
\end{proof}

\bigskip In the case $A+B=G$ we can make use of estimates for $A$ and $B$ from~\cite{Shp-Decomp}.
\begin{lemma}[Shparlinski]\label{lem:shp_estimate}
Let $p$ be a prime number. If for a subgroup $G \subseteq \mathbb{F}_p^{*}$ there is a nontrivial decomposition into some sets $A$ and $B$ then
\begin{equation}\label{Shp_est}
|G|^{1/2+o(1)}=\min(|A|,|B|) \le \max(|A|,|B|)=|G|^{1/2+o(1)}
\end{equation}
as $|G| \rightarrow \infty.$
\end{lemma}

\begin{corollary}\label{cor:subgroup_decomp_impossible}
Let $\eps \in (0,1]$ be a real number, $A, G \subset \mathbb{F}_p^{*}$ be sufficiently large multiplicative subgroups and $B \subseteq \mathbb{F}_p$ be an arbitrary nonempty set. If $|G \cap A| \ll |A|^{1-\eps}$ and $|G| \ll p^{1-\eps/6},$ then $G$ has no nontrivial representation as $G=A+B.$
\end{corollary}
\begin{proof}
If $G$ is sufficiently large and $|G| \ll p^{1-\eps/6},$ then applying estimates~\eqref{Shp_est} one can show that conditions of Theorem~\ref{th:decomposition_small_subgroups} are satisfied. Hence, if $A+B=G$ and $|G \cap A| \ll |A|^{1-\eps},|G| \ll p^{1-\eps/6}$ then by Lemma~\ref{lem:shp_estimate} and Theorem~\ref{th:decomposition_small_subgroups} we have $|B||A|^{1+\eps} \ll |G|^{1+\eps/2+o(1)} \ll |G|,$ a contradiction.
\end{proof}

\section{Stepanov's method in $\mathbb{Z}/p^2\mathbb{Z}$} \label{sec:stepanov}

Let $p$ be a prime number, $p\ge 3$.
Put
$$
    \G = \{ m^p ~:~ 1\le m \le p-1 \} \subseteq \Z_{p^2} \,.
$$
It is easy to see that $\G$ is a subgroup and that
$$
    \G = \{ x^p ~:~ x\in \Z^*_{p^2} \} = \{ x \in \Z^*_{p^2} ~:~ x^{p-1} \equiv 1 \pmod {p^2} \}
$$
because of $x\equiv y \pmod p$ implies $x^p \equiv y^p \pmod {p^2}$.
Further, one can check
$$
    \Z^*_{p^2} = \bigsqcup_{j=1}^p (1+pj) \G := \bigsqcup_{j=1}^p \xi_j \G\,,
$$
and $\Z_{p^2} \setminus \Z^*_{p^2} = \{ 0 \} \bigsqcup p\G$ (see \cite{Malykhin_p^2}).

\bigskip

Put
$$
    f(X) = X + \frac{X^2}{2} + \frac{X^3}{3} + \dots + \frac{X^{p-1}}{p-1} \in \Z_p [X] \,.
$$

Recall a lemma from \cite{Malykhin_p^2}.

\begin{lemma}
    Let $r\ge 2$ be a positive integer,
    and $R\subseteq \Z_{p^r}$ be a multiplicative subgroup,
    $|R|$ divides $p-1$.
    Then the natural projection $\_phi : \Z^*_{p^{r}} \to \Z^*_{p^{r-1}}$ is a bijection onto $R$
    and $\_phi (R)$ is a multiplicative subgroup of $\Z^*_{p^{r-1}}$
    (of size $|R|$).
\label{l:projection_p^k}
\end{lemma}

We need in a simple lemma.

\begin{lemma}
    Let $\la = (1+sp) g$, $s\in [p]$ and $g\in \G$.
    For all $i,j\in [p]$ the following holds
\begin{equation}\label{f:conv_-1}
    |\{ x-y \equiv \la \pmod {p^2} ~:~ x\in \xi_i \G,\, y\in \xi_j \G \}|
        =
\end{equation}
\begin{equation}\label{f:conv_-1'}
        =
            |\{ b \in \Z^*_p \,, b\neq g ~:~ f(b g^{-1} \pmod p) \equiv (i-j) b g^{-1} + j-s \} |
            =
\end{equation}
\begin{equation}\label{f:conv_1}
    =
    |\{ b \in \Z^*_p \,, b\neq g ~:~ f(b) \equiv (i-j) b + j-s \} |\,.
\end{equation}
Further
\begin{equation}\label{f:conv_2}
    |\{ x-y \equiv \la \pmod {p^2} ~:~ x\in \xi_i \G,\, y\in p \G \}| \le 1 \,,
\end{equation}
and
\begin{equation}\label{f:conv_2'}
    |\{ x-y \equiv \la \pmod {p^2} ~:~ x\in p \G,\, y\in \xi_j \G \}|
        \le 1 \,.
\end{equation}
\label{l:conv_simple}
\end{lemma}
\begin{proof}
Let us prove (\ref{f:conv_1}).
We construct one to one correspondence between the subset of $\Z^*_{p^2}$ from (\ref{f:conv_-1})
and the subset of $\Z^*_p$ from (\ref{f:conv_1}).
For some $1 \le m,n \le p-1$, we have
$$
    x-y \equiv (1+pi) m^p - (1+pj) n^p \equiv \la \pmod {p^2} \,.
$$
Thus $n \equiv m-g \pmod p$ and we obtain
$$
    (1+pi) m^p - (1+pj) n^p \equiv (1+pi) m^p - (1+pj) (m-g)^p
        \equiv
            \sum_{l=1}^p (-1)^{l-1} \binom{p}{l} g^l m^{p-l} + p (i m - j (m-g))
$$
\begin{equation}\label{tmp:20.01.2013_1}
    \equiv g-p g f(m g^{-1}) + p (i m - j (m-g))
                \equiv \la \pmod {p^2}
\end{equation}
as required.
In formula (\ref{tmp:20.01.2013_1}) we have used the fact that $g\in \G$ and hence $g^{p-1} \equiv 1 \pmod {p^2}$.

Further, suppose that for $m,n$ such that $1 \le m,n \le p-1$ the following holds
\begin{equation}\label{tmp:07.11.2012_1}
    (1+pi) m^p - p n^p \equiv \la \pmod {p^2} \,.
\end{equation}
Then $m\equiv g \pmod p$ and hence by Lemma \ref{l:projection_p^k} the number  $m$ is determined uniquely.
Substitution $m$ into (\ref{tmp:07.11.2012_1}) gives us $n \equiv (i-s)g \pmod p$.
Such $n$ does not exists if $(i-s) \equiv 0 \pmod p$ and $n$ is determined uniquely otherwise.
So, we have obtained (\ref{f:conv_2}).
Inequality (\ref{f:conv_2'}) follows similarly.
This completes the proof.
$\hfill\Box$
\end{proof}

\bigskip

Denote the sets from (\ref{f:conv_1}) as $M_{i,j} (\la)$
and from (\ref{f:conv_2}), (\ref{f:conv_2'}) as $M_{i,0} (\la)$, $M_{0,j} (\la)$, correspondingly.
Thus the previous lemma
represents
the sizes of such sets from $\Z_{p^2}$ via the sizes of some sets in $\Z_p$.
If $\la = 1$ then we write just $M_{i,j}$, $M_{i,0}$, and $M_{0,j}$.

To use Stepanov's method we need in a lemma from \cite{H}.

\begin{lemma}
    Let $r$ be a positive integer.
    Then there are two polynomials $q_r (X), h_r (X) \in \Z_p [X]$ such that
    $\deg q_r \le r+1$, $\deg h_r \le r-1$ and
    $$
        (X(1-X))^r \left( \frac{d}{dX} \right)^r f(X) = q_r (X) + (X^p - X) h_r (X) \,.
    $$
\label{l:f_derivative}
\end{lemma}

Thus we can convert the polynomial $f(X)$ of large degree (and its derivatives)
into $q_r (X)$, which has small degree.
Also we need in a lemma on linear independence of some family of polynomials.

\begin{lemma}
    Let $F(X,Y) \in \Z_p [X,Y]$ have degree less then $A$ with respect to $X$,
    and degree less then $B$ with respect to $Y$.
    Suppose that $AB \le p$ and $F$ is not vanish identically.
    Then $X^p$ does not divide $F(X,f(X))$.
\label{l:F_vanish}
\end{lemma}

Now we formulate the main result of the section.
We use Stepanov's method \cite{Stepanov}, \cite{H}, \cite{H-K}, \cite{K_Tula} in the proof
and include it for the sake of completeness.

\begin{proposition}
    Suppose that
    $Q,Q_1,Q_2 \subseteq \Z_{p^2}$
    are $\G$--invariant sets,
    $|Q| |Q_1| |Q_2| \ll p^5$, and
    $Q=\G$, and $Q_1$ is a coset over $\G$.
    Then
    \begin{equation}\label{f:H-K_2/3}
        \sum_{x\in Q} (Q_1 \c Q_2) (x) \ll p^{-1/3} ( |Q| |Q_1| |Q_2| )^{2/3} \,.
    \end{equation}
\label{p:H-K_2/3}
\end{proposition}
\begin{proof}
Let $s=|Q| |Q_1| |Q_2| /|\G|^3$.
Clearly, $s$ is a positive integer.
By Lemma \ref{l:conv_simple} (one can take the parameter $\la$ equals $1$)
to estimate the sum from (\ref{f:H-K_2/3}) we need to find an appropriate  upper bound for the size of the following set
$$
    M:= \bigcup_{l=1}^{s} M_{i_l,j_l} \,.
$$
Indeed, the sum of terms with cosets $p\G$  is negligible by estimates
(\ref{f:conv_2}), (\ref{f:conv_2'}) of Lemma \ref{l:conv_simple} and the assumption $s\ll p^2$.
Further, although the sets from (\ref{f:conv_-1}) are disjoint for different pairs $(\xi_i, \xi_j)$
their images (\ref{f:conv_1}) can intersects at most one point.
It is easy to see that the condition $Q=\G$ implies any three of such images cannot intersect.
Thus, we have
$$
    \sum_{x\in Q} (Q_1 \c Q_2) (x) \ll p| M | +
        p^{-1/3} ( |Q| |Q_1| |Q_2| )^{2/3} \,.
$$

Consider a polynomial $\Phi \in \Z_p [X,Y,Z]$ such that
$$
    \deg_X \Phi < A \,, \quad \deg_Y \Phi < B  \,, \quad \deg_{Z} \Phi < C \,.
$$
We have
\begin{equation}\label{f:Phi_pol}
    \Phi(X,Y,Z) = \sum_{a,b,c} \la_{a,b,c} X^a Y^b Z^{c} \,.
\end{equation}
Besides take
\begin{equation}\label{f:Psi_pol}
    \Psi (X) = \Phi (X,f(X), X^{p}) \,.
\end{equation}
Clearly
$$
    \deg \Psi < A + p (B+C) \,.
$$
If we will find the coefficients $\la_{a,b,c}$ such that, firstly,
the polynomial $\Psi$ is nonzero, and, secondly,
$\Psi$ has a root of order at least $D$ at any point of the set $M$
(except $0$ and $1$, may be) then
\begin{equation}\label{tmp:07.11.2012_2}
    |M|
        \ll
            (A + p (B+C) ) / D
\end{equation}
Thus, we should check that
$$
    \l( \frac{d}{d X} \r)^n \Psi (X) \Big|_{X=x} = 0 \,, \quad \forall n < D \,, \quad \forall x\in M \,.
$$
It is easy to see that for all $m,q$, $q\ge m$, and any $\mu$ the following holds
$$
    (X-\mu)^m \l( \frac{d}{d X} \r)^m (X-\mu)^q = \frac{q!}{(q-m)!} (X-\mu)^q \,.
$$
If $m>q$ then  the left hand side equals zero.
Using the last formula and Lemma \ref{l:f_derivative}
it is easy to check (or see \cite{H}, \cite{Malykhin_p^2}) that for any $x\in M_{i_l,j_l}$
one has
$$
    [X(1-X)]^n \l( \frac{d}{d X} \r)^n X^a f(X)^b X^{cp} \Big|_{X=x}
        =
            P_{n,l,a,b,c} (x) \,,
$$
where $P_{n,l,a,b,c} (X)$ is a polynomial of degree at most $A+B+C+2D$.
Whence for any $x\in M_{i_l,j_l}$, we have
$$
    [X(1-X)]^n \l( \frac{d}{d X} \r)^n \Psi(X) \Big|_{X=x}
        =
            P_{n,l} (x) \,,
$$
and each polynomial $P_{n,l}$ has at most $A+B+C+2D$ coefficients,
which are linear forms of $\la_{a,b,c}$.
Thus if
\begin{equation}\label{tmp:07.11.2012_3}
    sD (A+B+C+2D) < ABC
\end{equation}
then there are coefficients $\la_{a,b,c}$ not all zero such that the polynomials
$P_{n,l}$ vanish for all $n<D$ and all $l\in [s]$.

We choose the parameters $A,B,C$ and $D$ as
$$
    A=[ p^{2/3} s^{-1/3} ] \,, \quad B=C=  [ p^{1/3} s^{1/3} ] \,, \quad D = [p^{2/3} s^{-1/3} / 32] \,.
$$
The assumption $|Q| |Q_1| |Q_2| \ll p^5$ implies that $s\ll p^2$ and hence the choice is admissible.
Quick calculations show that the parameters satisfy condition (\ref{tmp:07.11.2012_3}).
Further, we have $AB\le p$ and by Lemma \ref{l:F_vanish} our polynomial $\Psi$ does not vanish identically.
Finally, substitution of the parameters into (\ref{tmp:07.11.2012_2}) gives the required bound.
This completes the proof.
$\hfill\Box$
\end{proof}

Previous versions of the result above can be found in \cite{H}, \cite{H-K}.
Variants for other groups
are contained
in \cite{K_Tula}, \cite{Malykhin_p^2}.

Using Proposition \ref{p:H-K_2/3}, one can easily deduce upper bounds for moments of convolution of $\G$.
 These estimates are the same as in the case of multiplicative subgroups in $\Z_p$ (see, e.g. \cite{ss}).

\begin{corollary}
    We have
    \begin{equation}\label{f:E_2_E_3}
        \E(\Gamma) \ll |\Gamma|^{5/2} \,, \quad \E_3 (\G) \ll |\G|^3 \log |\G| \,,
    \end{equation}
    and for all $l\ge 4$ the following holds
    \begin{equation}\label{f:E_l}
        \E_l (\G) = |\G|^l + O(|\G|^{\frac{2l+3}{3}}) \,.
    \end{equation}
\label{cor:E_l}
\end{corollary}

\begin{corollary}
    Let $d\ge 2$ be a positive integer.
    Arranging $(\G * \G) (\xi_1) \ge (\G * \G) (\xi_2) \ge \dots $, where $\xi_j \neq 0$
    belong to distinct cosets, we have
    $$
        (\G * \G) (\xi_j) \ll |\G|^{\frac{2}{3}} j^{-\frac{1}{3}} \,.
    $$
    Actually, one can take different cosets of $\G$ in bounds above.
\label{c:3_d_moment}
\end{corollary}




\section{On Heilbronn's exponential sum}\label{part-Helib}
\label{sec:proof}

We formulate a consequence of so--called the eigenvalues method
(see Proposition 28 from \cite{s_ineq} as well as the proof of Theorem 27 from \cite{s_mixed}).
This is a key new ingredient of our
proof.

\begin{lemma}
    Let $A\subseteq \Gr$ be a set,
    and let $\psi$ be a real even function with $\FF{\psi} \ge 0$.
    For any set $Q\subseteq A-A$ let $\psi^Q$ be the restriction of $\psi$ onto the set $Q$.
    Then
$$
     \frac{1}{|A|^3} \left( \sum_{x} \psi^Q (x) (A\c A)(x) \right)^3 \le \sum_{x,y,z\in A} \psi^Q (x-y) \psi^Q (x-z) \psi (y-z) \,.
$$
\label{l:ineq}
\end{lemma}

We need in a simple
lemma
about Fourier coefficients of an arbitrary set
which is invariant under the  action of a subgroup
(the case of the prime field can be found e.g. in \cite{ss}).

\begin{lemma}
        Let $\G \subseteq \Z^*_{p^2}$ be Heilbronn's subgroup,
        and $Q$ be an $\G$--invariant subset of $\Z^*_{p^2}$,
        that is $Q\G=Q$.
        Then for any $\xi \neq 0$ the following holds
\begin{equation}\label{f:G-inv_bound_F}
    | \FF{Q} (\xi) | \le \min \left\{ \left(\frac{|Q| p^2}{|\G|}\right)^{1/2} \,, \frac{|Q|^{3/4} p^{1/2} \E^{1/4} (\G)}{|\G|} \,,
                            \right\} \,.
\end{equation}
\label{l:G-inv_bound_F}
\end{lemma}
\begin{proof}
By $\G$--invariance it is easy to see that for any $\gamma \in \G$, we have
$$
    \FF{Q} (\xi) = \sum_x Q(x) e^{-2\pi i x\xi /p^2} = \sum_x Q(x) e^{-2\pi i x \gamma \xi /p^2} = \FF{Q} (\gamma \xi)
$$
and all numbers $\{ \gamma \xi \}_{\gamma \in \G}$ are different by the definition of Heilbronn's subgroup.
Thus the first bound follows from the Parseval identity and the second one from identity (\ref{f:energy_Fourier}).
This completes the proof.
$\hfill\Box$
\end{proof}

\bigskip

Using lemma above, we derive the following corollary.

\begin{corollary}
Let $S \subseteq \Z_{p^2}$ be any set.
Then
\begin{equation}\label{f:E_3(Q,G)}
    \E^*_3 (S,\G) \ll  |S| \log |\G| \cdot \left( \frac{|S| \E (\G)}{p^2} + p \E^{1/2} (\G) \right) \,.
\end{equation}
\label{c:E_3(Q,G)}
\end{corollary}
\begin{proof}
Put $L=\log |\G|$ and $\E = \E (\G)$.
By the pigeonhole principle, we have
\begin{equation*}\label{tmp:05.10.2013_1}
    \E^*_3 (S,\G) = \sum_{x \neq 0} (S\c S) (x) (\G \c \G)^2 (x)
         \le
            L \omega^2 \sum_x (S\c S) (x) \Omega (x) \,,
\end{equation*}
where $\omega$ is a real number and the set $\Omega$ is a set of the form
$$
    \Omega = \{ x \neq 0 ~:~ 2^{-1} \o < (\G \c \G) (x) \le \o \} \,.
$$
Clearly, $\Omega$ is a $\G$--invariant set and from the definition of Heilbronn's subgroup
it is easy to see that $\Omega \subseteq \Z^*_{p^2}$.
Thus, by the Fourier transform and Lemma \ref{l:G-inv_bound_F}, we get
$$
    \E^*_3 (S,\G)
        \ll
            L \omega^2 \frac{|S|^2 |\Omega|}{p^2}
                +
                    L \omega^2 |S| |\Omega|^{3/4} \E^{1/4} p^{-1/2} \,.
$$
By Corollary \ref{c:3_d_moment}, we obtain
\begin{equation}\label{tmp:05.10.2013_1}
    |\Omega| \ll \min\{ \E \o^{-2}, |\G|^3 \o^{-3} \} \,.
\end{equation}
Thus, substitution of (\ref{tmp:05.10.2013_1}) gives us
$$
    \E^*_3 (S,\G)
        \ll
            \frac{L|S|^2 \E}{p^2} + L |S| \E^{1/4} p^{-1/2} \cdot \min\{ \E^{3/4} \o^{1/2}, |\G|^{9/4} \o^{-1/4} \} \,.
$$
Optimizing over $\o$, we get
$$
    \E^*_3 (S,\G)
        \ll
            \frac{L|S|^2 \E}{p^2} + L |S| \E^{1/2} p
$$
as required.
$\hfill\Box$
\end{proof}

\bigskip

Now we can prove our main result.

\begin{theorem}
    Let $p$ be a prime number.
    Then
    \begin{equation}\label{f:subgroup_energy}
        \E (\G) \ll p^{\frac{32}{13}} \log^{\frac{14}{13}} p \,.
    \end{equation}
\label{t:Heilbronn_new}
\end{theorem}
\begin{proof}
Let $|\G|=t = p-1$, $\E = \E(\G) = |\G|^3 / K$, $\E_3 = \E_3 (\G)$, $L=\log |\G|$.
Applying estimate of Corollary \ref{c:3_d_moment}, we obtain
\begin{equation}\label{tmp:17.11.2012_1}
    2^{-2} \E \le \sum_{s ~:~ 2^{-1} |\G| K^{-1} < (\G \c \G) (s) \le cK} (\G \c \G)^2 (s) \,.
\end{equation}
Put
$$
    S_j = \{ s\in \G-\G ~:~ 2^{j-2} |\G| K^{-1} < (\G \c \G) (s) \le 2^{j-1} |\G| K^{-1} \} \,,
$$
where $j \in [l]$, $2^l\le 2c^{} K^2 |\G|^{-1} \ll K^2 |\G|^{-1}$.
Thus by (\ref{tmp:17.11.2012_1}) the following holds
$$
    2^{-2} \E \le \sum_{j=1}^l \sum_{s\in S_j} (\G \c \G)^2 (s) \,.
$$
By pigeonhole principle, we find $j\in [l]$ such that
\begin{equation}\label{tmp:17.11.2012_D_pred}
    2^{-2} l^{-1} \E  \le \sum_{s\in S_j} (\G \c \G)^2 (s) \le |S_j| (2^{j-1} |\G| K^{-1})^2 \,.
\end{equation}
Put $S=S_j$, $\Delta = 2^{j-1} |\G| K^{-1}$, and $g(x) = (\G\c \G) (x) S(x)$.
Applying Lemma \ref{l:ineq} with $A=\G$, $Q=S$, $\psi = \G\c \G$ and inequality (\ref{tmp:17.11.2012_D_pred}),
we obtain
\begin{equation}\label{tmp:19.07.2012_1}
    t^{-3} (2^{-2} l^{-1} \E)^3 \le \sum_{x,y,z\in \G} g(x-y) g(x-z) (\G\c \G) (y-z) \,.
\end{equation}
Further
\begin{equation}\label{tmp:17.11.2012_2}
    t^{-3} (2^{-2} l^{-1} \E)^3
        \le
            \sum_{\a,\beta} g(\a) g(\beta) (\G\c \G) (\a-\beta) \Cf_3 (\G) (\a,\beta) \,.
\end{equation}
It is easy to check that the term $\a=\beta$ in (\ref{tmp:17.11.2012_2}) is negligible,
because otherwise
$$
    t^{-3} (l^{-1} \E)^3 \ll t \E_3 (\G)
$$
and the result follows.
By the subgroup property, we have
\begin{equation}\label{tmp:new_1}
    \sum_{x\in \G} (g * \G)^2 (x) = t^{-1} \left( \sum_x g (x) (\G \c \G) (x) \right)^2 \,.
\end{equation}
Thus, the summation in (\ref{tmp:17.11.2012_2}) can be taken over $\a \neq \beta$ such that
$$
    (\G\c \G) (\a-\beta) \ge \frac{\E}{8l|\G|^2} := d \,.
$$
Thus
\begin{equation}\label{tmp:17.11.2012_5}
    2^{-7} l^{-3} t^{-3} \E^3
        \le
            \sum_{\a  \neq  \beta ~:~ (\G \c \G) (\a-\beta) \ge d} g(\a) g(\beta) (\G \c \G) (\a-\beta) \Cf_3 (\G) (\a,\beta) \,.
\end{equation}
By Cauchy--Schwartz inequality, Corollary \ref{cor:E_l}
and
formula
$$
    \sum_{\a,\beta} \Cf^2_3 (\G) (\a,\beta) = \E_3 \,,
$$
we get
$$
    l^{-6} t^{-6} \E^6
        \ll
            \E_3
                \D^4
                    \sum_{\a \neq \beta ~:~ (\G\c \G) (\a-\beta) \ge d}
                        S(\a) S(\beta) (\G\c \G)^2 (\a-\beta)
                            \ll
                                \D^4 \E^*_3 (S,\G) \,.
$$
Using Corollary \ref{c:E_3(Q,G)}, we obtain
\begin{equation}\label{tmp:05.10.2013_1}
    l^{-6} t^{-6} \E^6
        \ll
            \E_3 \D^4 |S| L \cdot \left( \frac{|S| \E}{p^2} + p \E^{1/2} \right) \,.
\end{equation}
If the first term in (\ref{tmp:05.10.2013_1}) dominates then we are done
in view of the inequality $\D^2 |S| \le \E$.
Otherwise
$$
    l^{-6} t^{-6} \E^6
        \ll
            \E_3 \D^4 |S| L p \E^{1/2} \,.
$$
Accurate computations, using (\ref{tmp:17.11.2012_D_pred}) show
$$
    \left( \frac{\E}{t L} \right)^{5}
        \ll L^2 \D^{2} t^{5} \E^{1/2} \,.
$$
Applying estimate $\D \ll K$ after some calculations we obtain the result.
This completes the proof.
$\hfill\Box$
\end{proof}

\bigskip

Using
accurate  arguments from \cite{K_Tula}
and
appropriate generalizations of estimate from
Corollary \ref{c:3_d_moment}
one can, certainly,
obtain
similar
bounds for $\T_k (\G)$ with large $k$.
We do not make such calculations.

\bigskip

\begin{remark}
To obtain (\ref{tmp:new_1}) we have used the fact that $\G$ is a subgroup.
For general set $A$ a similar inequality takes place.
Indeed, let $g$ be a real even function and $A$ be a set.
In terms of paper \cite{s_mixed} (or see \cite{s_ineq}), we have
$
    (\oT^g_A)^2 (x,y) = \Cf_3 (A^c,g,g) (x,y)
$
and, hence,
$$
    \mu^2_0 (\oT^g_A) \ge |A|^{-1} \langle (\oT^g_A)^2 A, A \rangle = |A|^{-1} \sum_{x\in A} (A*g)^2 (x) \,.
$$
Thus, the arguments in lines (\ref{tmp:19.07.2012_1})---(\ref{tmp:17.11.2012_5}) take place in general.
\end{remark}

\bigskip

Theorem above implies a result on exponential sums over subgroups in $\Z^*_{p^2}$
(see details of the proof in \cite{H,H-K} or \cite{s_heilbronn}).

\begin{corollary}
    Let $p$ be a prime, $a\neq 0 \pmod p$,
    and $M,N$ be positive integers, $N\le p$.
    Then
    \begin{equation}\label{f:main_progr}
        \left| \sum_{n=M}^{N+M} e\left( \frac{an^p}{p^2} \right) \right|
            \ll
                p^{\frac{8}{13}} N^{\frac{1}{4}} \log^{\frac{14}{42}} p \,.
    \end{equation}
    In particular
    \begin{equation}\label{f:main_S}
        |S(a)| \ll p^{\frac{45}{52}} \log^{\frac{14}{42}} p \,.
    \end{equation}
\label{cor:main}
\end{corollary}
{\it Sketch of the proof.~}
We have
$$
    |S(a)| \le \E^{1/4} (\G) N^{1/4}
$$
and the result follows. $\hfill\Box$

\bigskip

Using the arguments from \cite{BFKS} and Theorem \ref{t:Heilbronn_new},
we obtain the following result
about Fermat quotients.
By $l_p$ denote the smallest $n$ such that $q(n) \neq 0 \pmod p$.
In \cite{BFKS} an upper bound for $l_p$ was obtained.

\begin{theorem}
    One has
    $$
        l_p \le (\log p)^{\frac{463}{252} + o(1)}
    $$
    as $p\to \infty$.
\label{t:l_p}
\end{theorem}

In \cite{s_heilbronn} we found
an
estimate for the additive energy of $\G$
which allows improve Theorem \ref{t:l_p}.
Namely, we got
$$
        l_p \le (\log p)^{\frac{7829}{4284} + o(1)}
$$

\bigskip

Now we formulate our new result on upper bound for $l_p$.

\begin{theorem}
        One has
    $$
        l_p \le (\log p)^{\frac{5977}{3276} + o(1)}
    $$
    as $p\to \infty$.
\label{t:Fermat_quatients}
\end{theorem}
{\it Sketch of the proof.~}
Let $l_p = (\log p)^{\kappa + o(1)}$, $\kappa > 0$.
Let also $k<p^2$ be a positive integer and put $N (k)$
be the number solutions of the congruence
$$
    ux \equiv y \,,\quad \quad 0< |x|, |y| \le p^{2+o(1)} k^{-1} \,, \quad \quad u\in \G \,.
$$
By the arguments of paper \cite{BFKS} the number $\kappa$ can be estimated, very roughly,
from the formula $k = p^{\kappa}$,
where $k$ is the smallest number such that inequality
\begin{equation}\label{tmp:10.11.2012_1}
    \frac{k}{p} \gg \frac{k}{p^3} \cdot \left( \frac{p^9 \E(\G) N(k)}{k^2} \right)^{1/4}
\end{equation}
holds.
To estimate $N(k)$ we use Lemma 9 from \cite{BFKS} which gives for any positive integer $\nu$ that
$$
    N(k) \ll (p^{2+o(1)} k^{-1} ) p^{1/(2\nu(\nu+1))}  + (p^{2+o(1)} k^{-1} )^2 p^{-1/\nu} \,.
$$
Choosing $\nu = 6$, we find
in the range of parameter $k$ that
$$
    N(k) \ll (p^{2+o(1)} k^{-1} ) p^{1/84} \,.
$$
Substituting the last estimate into (\ref{tmp:10.11.2012_1}),
we obtain the result.
$\hfill\Box$

\bigskip

Note that
$$
    \frac{7829}{4284} = 1.82749 \dots \dots \quad \quad \mbox{ and} \quad \quad
    \frac{5977}{3276} = 1.82448 \dots \dots
$$
It was conjectured by A. Granville (see \cite{Granville_Fermat}, Conjecture 10) that
$$
    l_p = o ( (\log p)^{\frac{1}{4}} )
$$
and H. W. Lenstra \cite{Lenstra} conjectured that, actually, $l_p \le 3$.

\bigskip

Theorem \ref{t:Fermat_quatients}
has a consequence (see \cite{Lenstra}).

\begin{corollary}
    For every $\eps > 0$ and a sufficiently large integer $n$,
    if $a^{n-1} \equiv 1 \pmod n$
    for every positive integer $a \le (\log p)^{\frac{5977}{3276} + \eps}$
     then $n$ is squarefree.
\end{corollary}

Discussion and further applications can be found in \cite{s_heilbronn}.

\section{Concluding remarks}
\label{sec:conclusion}

At the end of the paper we make several remarks about possible extensions
of our results onto the groups $\Z^*_{p^{k}}$, $k\ge 1$.

We begin with the problem of estimation of exponential sums over multiplicative subgroups of such groups.
It can be shown (see \cite{Malykhin_p^k}) that if $\G \subseteq \Z^*_{p^{k}}$ is a
subgroup
and $p$ divides $|\G|$ then the exponential sum over $\G$ vanishes.
Thus a question about the estimation of exponential sums
is trivial
in the case.
If $|\G|$ divides $p-1$
then the exponential sum can be reduced to the cases
of subgroups in $\Z^*_p$ and $\Z^*_{p^2}$ (see the main result from \cite{Malykhin_p^k}).
The reason is the existence of the natural projection $\_phi : \Z^*_{p^{k}} \to \Z^*_{p^{k-1}}$, $k\ge 2$
which is defined by the rule $\_phi (x) \equiv x \pmod {p^{k-1}}$,
see Lemma \ref{l:projection_p^k}.

The projection $\_phi$ allows estimate quantities $\T_l (\G)$, $\G \subseteq \Z^*_{p^{k}}$
via quantities $\T_l (\_phi(\G))$ of subgroups from  $\Z^*_{p^{k-1}}$.
In $\Z_{p^2}$ an adaptation of Stepanov's method from \cite{Malykhin_p^2}
gives
such estimates directly, provided by $|\G|$ divides $p-1$.
The existence of Stepanov's estimates similar Proposition \ref{p:H-K_2/3}
allows to apply the method from section \ref{sec:proof} to obtain
better bounds.
We do not make such calculations.

\bigskip

\no{Shkredov I.D.\\
Division of Algebra and Number Theory,\\ Steklov Mathematical
Institute,\\
ul. Gubkina, 8, Moscow, Russia, 119991}
\\
and
\\
Delone Laboratory of Discrete and Computational Geometry,\\
Yaroslavl State University,\\
Sovetskaya str. 14, Yaroslavl, Russia, 150000
\\
and
\\
Institute for Information Transmission Problems RAS,  \\
Bolshoy Karetny per. 19, Moscow, Russia, 127994\\
{\tt ilya.shkredov@gmail.com}

\bigskip

\no{Solodkova E.V.\\
National Research University Higher School of Economics, \\
{\tt hsolodkova@gmail.com}}

\bigskip

\no{Vyugin I.V.\\
Insitute for Information Transmission Problems RAS,  \\
Bolshoy Karetny per. 19, Moscow, Russia, 127994
\\
and
\\
National Research University Higher School of Economics,\\
Vavilova Str., 7, Moscow, Russia, 117312\\
{\tt vyugin@gmail.com}}.

\end{document}